\documentclass{amsart}
\usepackage[utf8]{inputenc}
\usepackage{amssymb,amsfonts, stmaryrd, amsmath, amsthm, enumerate, mathtools, hyperref}
\usepackage[all]{xy}
\usepackage{wrapfig,tikz, tikz-cd}
\usetikzlibrary{arrows}
\usepackage{color}

\RequirePackage{tikz-cd}
\RequirePackage{amssymb}
\usetikzlibrary{calc}
\usetikzlibrary{decorations.pathmorphing}
\tikzset{curve/.style={settings={#1},to path={(\tikztostart)
    .. controls ($(\tikztostart)!\pv{pos}!(\tikztotarget)!\pv{height}!270:(\tikztotarget)$)
    and ($(\tikztostart)!1-\pv{pos}!(\tikztotarget)!\pv{height}!270:(\tikztotarget)$)
    .. (\tikztotarget)\tikztonodes}},
    settings/.code={\tikzset{quiver/.cd,#1}
        \def\pv##1{\pgfkeysvalueof{/tikz/quiver/##1}}},
    quiver/.cd,pos/.initial=0.35,height/.initial=0}
\tikzset{tail reversed/.code={\pgfsetarrowsstart{tikzcd to}}}
\tikzset{2tail/.code={\pgfsetarrowsstart{Implies[reversed]}}}
\tikzset{2tail reversed/.code={\pgfsetarrowsstart{Implies}}}
\tikzset{no body/.style={/tikz/dash pattern=on 0 off 1mm}}

\newtheorem{theorem}{Theorem}[section]
\newtheorem{lemma}[theorem]{Lemma}
\newtheorem{proposition}[theorem]{Proposition}
\newtheorem{corollary}[theorem]{Corollary}
\newtheorem{conjecture}[theorem]{Conjecture}
\newtheorem{theorema}{Theorem}

\theoremstyle{definition}
\newtheorem{definition}[theorem]{Definition}
\newtheorem{example}[theorem]{Example}
\newtheorem{remark}[theorem]{Remark}

\def\h{\mathfrak{h}}
\def\g{\mathfrak{g}}

\def\n{\mathfrak{n}}
\def\t{\mathfrak{t}}

\def\u{\mathfrak{u}}

\def\gl{\mathfrak{gl}}
\def\sl{\mathfrak{sl}}

\def\L{\mathfrak{L}}

\def\r{\mathfrak{r}}

\DeclareMathOperator{\Out}{Out}
\DeclareMathOperator{\Aut}{Aut}
\DeclareMathOperator{\HH}{HH}
\DeclareMathOperator{\Hom}{Hom}
\DeclareMathOperator{\rank}{rank}

\DeclareMathOperator{\Der}{Der}
\DeclareMathOperator{\Inn}{Inn}
\DeclareMathOperator{\MT}{MT}
\DeclareMathOperator{\mt}{mt}
\renewcommand{\Im}{\mathrm{im}}

\newcommand{\pirank}{\pi_1\!\text{-}\!\rank}
\newcommand{\mtrank}{{\rm mt}\text{-}\!\rank}

\title[Maximal Tori in $\HH^1$ and the fundamental group]{Maximal Tori in $\HH^1$ and the fundamental group}

\author{Benjamin Briggs}
\address{
University of Utah, Department of Mathematics,
Salt Lake City, Utah 84112}
\email{briggs@math.utah.edu}

\author{Lleonard Rubio  y Degrassi}
\address{Dipartimento di Informatica - Settore di Matematica, Universit\`{a} degli Studi di Verona, Strada le Grazie 15 - Ca’ Vignal, I-37134 Verona, Italy}
\email{lleonard.rubioydegrassi@univr.it}

\thanks{\noindent
{\it Funding.}
The second author has been  supported by the Fundaci\'on ‘S\'eneca’ of Murcia [19880/GERM/15]; the INdAM postdoctoral research grant 2019-2020; and by the project PRIN 2017 - Categories, Algebras: Ring-Theoretical and Homological Approaches.\\
{\it Acknowledgements.}
The first author would like thank Anne Dranowski for explaining some 
facts about the Lie theory of algebraic groups, and Vincent G\'elinas for making some important suggestions. The second author would like to thank Alexandra Zvonareva for explaining some facts about the outer automorphism group of Brauer graph algebras. We are thankful to Markus Linckelmann for 
many helpful correspondences. We also would like to thank Hao Chang for pointing out a mistake in Example 4.5. Lastly, we thank Manuel Saor\'in for 
his support and for pointing out the map from the dual of the 
fundamental group to the first Hochschild cohomology.
}

\begin{document}

\subjclass[2010]{16E40, 16D90 (Primary) 17B50, 13D03, 18D50 (Secondary)}
\keywords{Hochschild cohomology, invariance, fundamental group, monomial algebras, }

\begin{abstract}
We investigate maximal tori in the Hochschild cohomology Lie algebra $\HH^1(A)$ of a finite dimensional algebra $A$, and their connection with the fundamental groups associated to presentations of $A$.  We prove that every maximal torus in $\HH^1(A)$ arises as the dual of some fundamental group of $A$, extending work of Farkas, Green and Marcos; de  la  Pe\~na  and Saor\'in; and Le Meur.  Combining this with known invariance results for Hochschild cohomology, we deduce that (in rough terms) the largest rank of a fundamental group of $A$ is a derived invariant quantity, and among self-injective algebras, an invariant under stable equivalences of Morita type.  Using this we prove that there are only finitely many monomial algebras in any derived equivalence class of finite dimensional algebras; hitherto this was known only for very restricted classes of monomial algebras.
\end{abstract}

\maketitle

\section*{Introduction}

The first Hochschild cohomology Lie algebra $\HH^1(A)$, along with all of the Lie theoretic invariants which may derived from it, is an important object attached to any finite dimensional algebra $A$, 
and this structure has been studied intensely in recent years; see for example \cite{ALS,BKL,CSS,ER,LR1,RSS,Stra2,Taillefer}. In particular, authors have been interested in properties such as solvability and nilpotence of $\HH^1(A)$, for specific classes of algebras. The notion of a maximal torus plays a fundamental role in the theory of Lie algebras---for example, they are a key ingredient in the classification of semi-simple Lie algebras in characteristic zero. In this work we improve our understanding of the Lie structure of $\HH^1(A)$ by classifying its maximal tori in terms of presentations of $A$.

Classically, the bracket on $\HH^1(A)$ was used to understand the deformation theory of $A$. Within representation theory, since this Lie algebra is a derived invariant of $A$---and, among self-injective algebras, an invariant under stable equivalences of Morita type---its structure is often used in classification problems. More recently, by studying the closely related algebraic group of outer automorphisms of $A$, maximal tori have been used to impressive effect to obtain complete combinatorial derived invariants for gentle algebras \cite{AG} and Brauer graph algebras \cite{AZ}, and they have been used in modular representation theory of blocks \cite{Kessar}. Beyond this, however, fine Lie theoretic properties of $\HH^1(A)$ are not often used in reverse, to help understand the original algebra $A$. This work contains a number of results in this direction. In particular, we show that the structure of $\HH^1(A)$ can be used to deduce information about the shape of the Gabriel quiver of $A$ (see Theorem \ref{ta_betti}).

The second main pillar of this paper concerns the homotopy theory of bound quivers. The fundamental group $\pi_1(Q,I)$ of a bound quiver $(Q,I)$ was introduced in \cite{MP} using the covering techniques developed in \cite{BG,Rie}. Starting with an algebra $A$, one obtains different groups $\pi_1(Q,I)$ if one varies the presentation $A\cong kQ/I$. While this has previously been considered a defect of the theory, our perspective is that one should consider the moduli of all fundamental groups, much as one considers the moduli of all tori in a Lie algebra. Indeed, our results show that tori in $\HH^1(A)$ are intimately linked with the fundamental groups of $A$.

A diagonalizable subalgebra of $\HH^1(A)$ is by definition a subspace generated by derivations which act simultaneously diagonalizably on $A$. These have been studied in \cite{DS,FGM,Le}, and it is now well-known that they are connected with the spaces $\Hom(\pi_1(Q,I),k^+)$ of additive characters on the fundamental groups  
of $A$. The next result was proven by Farkas, Green and Marcos under certain technical assumptions \cite[Theorem 3.2]{FGM}, and later proven for triangular algebras by Le Meur  \cite[Proposition 2.6]{Le}; our result applies to all finite dimensional algebras.

\begin{theorema}
\label{main1}
Let $A$ be a finite dimensional algebra over a field $k$. Every maximal diagonalizable subalgebra of $\HH^1(A)$ is isomorphic to $\Hom(\pi_1(Q,I),k^+)$ for some 
presentation $A\cong kQ/I$.
\end{theorema}

The isomorphisms appearing in the theorem are realized by  concretely defined maps $\Hom(\pi_1(Q,I),k^+)\to \HH^1(A)$; for triangular algebras this map was introduced in \cite{AP}, and it was generalized to arbitrary finite dimensional algebras in \cite{DS}.

At this point we involve some Lie theory, proving that the maximal diagonal subalgebras of $\HH^1(A)$ are exactly the maximal tori of $\HH^1(A)$. With this in mind, we define $\mtrank(\HH^1(A))$ to be the maximal dimension of a torus in $\HH^1(A)$, and we define $\pirank(A)$ to be the maximal dimension of a dual fundamental group $\Hom(\pi_1(Q,I),k^+)$ for some (minimal) presentation of $A$. We immediately deduce the next result.

\begin{theorema}
\label{thm2}
For any finite dimensional $k$-algebra $A$ we have 
\[
\mtrank(\HH^1(A))=\pirank(A).
\]
\end{theorema}

The \emph{Lie-theoretic} quantity on the left-hand-side is a derived invariant of $A$, and so we deduce from this theorem that the \emph{combinatorial} quantity on the right-hand-side is a derived invariant of $A$. Similarly, among self-injective algebras, we deduce that $\pirank(A)$ is an invariant under stable equivalences of Morita type. As a first application of this idea we bound the maximal toral rank of $\HH^1(A)$ in terms of the first Betti number of the Gabriel quiver of $A$ (that is, the number of holes in the underlying graph). 

\newpage 
\begin{theorema}
\label{ta_betti}
If $A$ is a finite dimensional algebra over an algebraically closed field $k$, with Gabriel quiver $Q$, then 
\[
\mtrank(\HH^1(A))\leq \beta_1(Q).
\]
If $A$ is a monomial (or semimonomial) algebra then equality holds.
\end{theorema}
We note that lower bounds on the toral rank come from exhibiting tori, while upper bounds are in general more difficult to come by. The fact that equality holds for monomial algebras has a striking application.

\begin{theorema}
\label{Main2}
If two finite dimensional monomial algebras (over any field) are derived equivalent, then their Gabriel quivers  contain the same number of arrows. Consequently, there are only finitely many monomial algebras in any given derived equivalence class.
\end{theorema}

It was proven by Avella-Alaminos and Geiss that among gentle algebras the number of arrows is a derived invariant \cite[Proposition B]{AG}. 
We are not aware of any other large classes of monomial algebras for which Theorem \ref{Main2} was known, beyond the obvious cases.

Theorem \ref{ta_betti} also implies (and extends to semimonomial algebras) a result of Bardzell and Marcos: if $A = kQ/I$ is a finite dimensional monomial algebra for which $\HH^1(A) = 0$, then $Q$ is a tree {\cite[Theorem 2.2]{BM}}.

The semimonomial algebras appearing in Theorem \ref{thm2} are defined in Section \ref{s_pi1}. Examples occur frequently, including all commutative monomial algebras and all quantum complete intersections (see Subsection \ref{s_semimon}). Many authors have studied the Hochschild cohomology Lie algebra of quantum complete intersections, and calculations are notoriously complicated, with the structure depending finely on the quantum parameters and exponents. It is therefore surprising that we are able to prove in all cases that $\mtrank(\HH^1(A_q))$ is equal to the number of variables generating the quantum complete intersection $A_q$ (see Corollary \ref{qci}). The picture in the modular representation theory is very similar: we prove that if $k$ is a field of characteristic $p$ and $G$ is a $p$-group, then $\mtrank(\HH^1(kG))$ is equal to minimal number of generators for $G$ (see Proposition \ref{p_pgroup}). 

Another well-studied class of algebras we apply our results to are the simply connected algebras, that is, triangular algebras for which the fundamental group of any presentation is trivial.

\begin{theorema}
\label{t_simplycon}
If $A$ is a simply connected finite dimensional algebra over an algebraically closed field, then $\HH^1(A)$ is nilpotent.
\end{theorema}

One new insight in this paper in that we must consider non-admissible presentations to prove Theorem \ref{main1}. When working in positive characteristic these presentations are unavoidable because of the existence of derivations which do not preserve the radical (see Example \ref{e_JW}). Thus, our setup is necessarily more general than previous works on the fundamental group; this has the benefit that we allow all finite dimensional algebras directly, including non-basic algebras (see Example \ref{e_group}).

\subsection*{Outline} 
In Section \ref{s_pi1} we recall the definition of the fundamental group of a bound quiver. In Section \ref{s_HHLie} we discuss diagonalizable derivations and diagonalizable Lie algebras of $\HH^1(A)$. Section \ref{Maximal_tori} summarises the Lie theory needed to discuss maximal tori in $\HH^1(A)$. This section divides according to the characteristic of $k$: in characteristic zero we must consider algebraic groups, while in postive characteristic we must consider restricted Lie algebras. In Section \ref{s_toripi} we prove Theorems \ref{main1} and \ref{thm2} characterising maximal tori in terms of the fundamental group. In Section \ref{s_examples} we apply our results to various families of finite dimensional algebras including monomial algebras, commutative algebras, quantum complete intersections, simply connected algebras, and Kronecker chains. Here we prove Theorems \ref{ta_betti}, \ref{Main2} and \ref{t_simplycon}, and recover some known theorems along the way.

\section{The fundamental group of a bound quiver}\label{s_pi1}

Let $A$ be a finite dimensional algebra over an algebraically closed field $k$.  

A quiver is a directed graph $Q$ with a set $Q_0$ of vertices and a set $Q_1$ of arrows. A presentation is a surjective $k$-algebra homomorphism $\nu\colon kQ\to A$ from the path algebra of a finite quiver $Q$. We will fix a complete set $e_1,...,e_n$ of orthogonal idempotents for $A$ 
and assume all of our presentations induce a bijection between $Q_0$ and $\{e_1,...,e_n\}$.  We do \emph{not} assume that $I=\ker(\nu)$ is an admissible ideal of $kQ$ (see Remark \ref{r_nonadmissible}), and we do \emph{not} assume that $e_i$ are primitive idempotents (See Example \ref{e_group}).

The elements of $I$ are called \emph{relations},
and a \emph{minimal relation} in $I$ is a nonzero relation 
$r = \sum_{i=1}^s a_ip_i$,
where the $p_i$ are distinct paths in $Q$ and $a_i \in k\smallsetminus \{0\}$, such that there is no proper nonempty
subset $T \subset \{1, \dots , s\}$ for which $\sum_{i\in T}  a_i p_i \in I $. Note that all paths appearing in a minimal relation $r$ automatically share the same source and target, that is, are parallel.

If $\alpha\in Q_1$, then the \emph{formal inverse} of $\alpha$, denoted by $\alpha^{-1}$, is an arrow with source and target equal to the target and
the source of $\alpha$, respectively.
They assemble into a quiver $Q^{-1}$ with $Q^{-1}_0=Q_0$ and $Q^{-1}_1:=\{\alpha^{-1} | \alpha \in
Q_1\}$. Then we can form the double quiver $\overline{Q}$ where
$\overline{Q}_0 = Q_0$
and $\overline{Q}_1 = Q_1 \cup Q^{-1}_1$.
With this notation, a \emph{walk} in $Q$ is  an oriented path in $\overline{Q}$.

\begin{definition}
\label{homotopydef}
To define the fundamental group, we first recall the \emph{homotopy relation}, which was introduced in \cite{MP}. By definition $\sim_I$ is the equivalence relation 
on the set of walks in ${Q}$ generated by:
\begin{enumerate}
\item\label{hrel1} $\alpha^{-1}\alpha \sim_I e_i$ and $\alpha\alpha^{-1}\sim_I e_j$ for any arrow $\alpha$ with source $e_i$ and target $e_j$;
\item\label{hrel2} if $v \sim_I v^{\prime}$ then also $wvu \sim_I wv'u$, where $w,v,v'$ and $u$ are walks such that the concatenations $wvu$ and $wv^{\prime}u$
are well-defined;
\item\label{hrel3} $u \sim_I v$ if $u$ and $v$ are paths which occur with a nonzero coefficient in the same minimal relation.
\end{enumerate}

Fixing one of the chosen idempotents $e_i$, the set of $\sim_I$-equivalence classes of walks with source and target $e_i$ is denoted
by $\pi_1(Q,I,e_i)$. Concatenation of walks endows this set with a group structure whose unit is the equivalence class of the trivial walk $e_i$ \cite{MP}. The group $\pi_1(Q,I,e_i)$ is called the \emph{fundamental group} of $(Q,I)$ based at $e_i$.

If $Q$ is connected, another choice of idempotent $e_i$ will yield an isomorphic fundamental group,  so we adopt the simplified notation $\pi_1(Q,I)=\pi_1(Q,I,e_i)$. If $Q$ is not connected, then following the convention in \cite{DS}, the fundamental group  $\pi_1(Q,I)$ is the direct product of the fundamental groups of
the quivers with relations obtained by restricting to the connected components of $Q$.
\end{definition}

The fundamental group $\pi_1(Q, I)$ does depend on the ideal $I=\ker(\nu)$ and therefore
it is not an invariant of the isomorphism class of $A=kQ/I$; see for example \cite[Example 1.2]{Le}.

\begin{remark}\label{r_nonadmissible}
To prove our main theorem it will be important to consider non-admissible (but still minimal) presentations, such as $kQ=k[x]\to A=k[x]/(x^p-1)$ when $k$ has characteristic $p$. See Remark \ref{r_admisschar} and Example \ref{e_JW} for more detail. The usual proof that $\pi_1(Q, I)$ is a group extends easily to this context.
\end{remark}

\begin{example}\label{e_group}
Since we allow non-admissible presentations our algebras need not be basic. For example, let $G$ be a finite group, and let $Q$ be the quiver with a single vertex and with edges in bijection with the elements of $G$. Consider the natural presentation $\nu\colon kQ\to kG$, with $I=\ker(\nu)$. Then the group $\pi_1(Q,I)$ is canonically isomorphic to $G$.
\end{example}

\begin{definition}
We will write $\pi_1(Q,I)^\vee=\Hom(\pi_1(Q,I),k^+)$ for the group of additive characters on $\pi_1(Q,I)$. This will play a central role below.
\end{definition}

\begin{remark}
When $X$ is a based, connected topological space, there is a Hurewicz isomorphism $\pi_1(X)^\vee\cong H^1(X;k)$ between the group of additive characters on the fundamental group of $X$ and the first singular cohomology of $X$. Thus $\pi_1(Q,I)^\vee$ is analogous to the first singular cohomology of $(Q,I)$. 

In fact, from $(Q,I)$ one can produce a based space $X$ such that $\pi_1(X)\cong \pi_1(Q,I)$; see \cite{Bust} and also \cite{Rey}.
\end{remark}

In a similar vein, one can also view $\pi_1(Q,I)^\vee$ as a generalization of graph cohomology. 
Suppose that $Q$ has $n$ vertices 
and $m$ edges and $c$ connected components. The first Betti number of the underlying graph $|Q|$ of $Q$ is given by $\beta_1(|Q|)=\dim_kH^1(|Q|;k)=m -n + c$ \cite[Lemma 8.2]{De}. Because of this, we adopt the notation $\beta_1(Q)=m -n + c$. Intuitively, $\beta_1(Q)$ is the number of holes in $Q$.

\begin{definition}
An ideal $I$ in $kQ$ is called \emph{semimonomial} if, within any minimal relation $r = \sum_{i=1}^s a_ip_i$, $a_i\neq 0$, each path $p_i$ contains exactly the same arrows, occurring the same number of times (so the different $p_i$ are permutations of each other); see \cite[Section 1]{GAS}. Also see Subsection \ref{s_semimon} for examples.
\end{definition}

\begin{lemma}
\label{ineq}
For a bound quiver $(Q,I)$ we have that 
\[
\dim_k \pi_1(Q,I)^\vee\leq \beta_1(Q).
\]
Equality holds if $I$ is a monomial ideal, and more generally if $I$ is semimonomial.
\end{lemma}

\begin{proof}
Note that $H_1(|Q|;\mathbb{Z})$ is a free abelian group of rank $\beta_1(Q)$, and the abelianization $\pi_1(Q,I)^{\rm ab}$ is a quotient of $H_1(|Q|;\mathbb{Z})$. Indeed, if one chooses loops $\ell_1,...,\ell_{\beta_1(Q)}$ based at $e_i$ representing the generators of $H_1(|Q|;\mathbb{Z})$ (i.e.\ a loop for each  hole in $|Q|$), then by conditions (\ref{hrel1}) and (\ref{hrel2}) in Definition \ref{homotopydef} $\pi_1(Q,I)^{\rm ab}$ is generated by $\ell_1,...,\ell_{\beta_1(Q)}$. It follows that $\dim_k \pi_1(Q,I)^\vee\leq \beta_1(Q)$.

In the case that $I$ is monomial, condition (\ref{hrel3}) in Definition \ref{homotopydef} does not identify any walks, so the homotopy relation $\sim_I$ does not even depend on $I$, and $H_1(|Q|;\mathbb{Z})\cong \pi_1(Q,I)^{\rm ab}$.

If $I$ is semimonomial then all walks which are $\sim_I$-equivalent to the identity are in the commutator subgroup of $\pi_1(Q,(0))$, so again in this case  $H_1(|Q|;\mathbb{Z})\cong \pi_1(Q,I)^{\rm ab}$.
\end{proof}

\begin{remark}\label{r_semimon_converse}
If $k$ has characteristic zero then the proof of the lemma is reversible: $\dim_k \pi_1(Q,I)^\vee = \beta_1(Q)$ if and only if $I$ is semimonomial. 

To obtain a converse in characteristic $p$, one could define $I$ to be $p$-\emph{semimonomial} if for each path in a given minimal relation, each arrow occurs the same number of times modulo $p$. Then in this context $\dim_k \pi_1(Q,I)^\vee = \beta_1(Q)$ if and only if $I$ is $p$-semimonomial.  
\end{remark}

\begin{remark} 
While special presentations will produce interesting fundamental groups, the fundamental group is very often trivial, because in a generic presentation the minimal relations will involve many terms.  Le Meur \cite{Le} considers the relation between all the different fundamental groups, along with some natural surjections between them. It is proven in loc.\ cit.\ that for certain very specific classes of algebras there is one fundamental group which surjects onto all others; see  \cite[Theorem 1]{Le}. In general this is not the case.
\end{remark}

\section{The Hochschild cohomology Lie algebra}
\label{s_HHLie}

As before, let $A$ be a finite dimensional algebra over an algebraically 
closed field $k$, with a complete set of orthogonal idempotents $e_1,...,e_n$.

A \emph{derivation} on $A$ is a $k$-linear map $f \colon A\to A$ satisfying the Leibniz rule, that is, 
$f(ab)=$ $f(a)b+af(b)$ for all $a,b\in A$. The {space of derivations} $\Der(A)$ on $A$ is a Lie algebra: if $f$ and $g$ are 
derivations on $A$, then so is $[f,g]:=$ $f\circ g-g\circ f$. 
The space $\Inn(A)$ of \emph{inner derivations} consists of those derivations of the form $[c,-]\colon a \mapsto ca-ac$,
for $c\in$ $A$. Then $\Inn(A)$ is a Lie ideal in $\Der(A)$ and there is a canonical isomorphism 
\[
\HH^1(A)\cong \Der(A)/\Inn(A),
\]
which we take as our definition of the first Hochschild cohomology group of $A$. A slightly more efficient description of this group will be useful: 

\begin{definition}
Let $\Der_0(A)$ be the space of those $k$-linear derivations $f$ such that $f(e_i)=0$ for $i=1,...,n$, and let $\Inn_0(A)= \Der_0(A)\cap \Inn(A)$. By  \cite[Proposition 1]{DS} there is a canonical isomorphism
\[
\HH^1(A)\cong \Der_0(A)/\Inn_0(A).
\]
\end{definition}

\begin{definition}

An element $f \in \HH^1(A)$ 
is called \emph{diagonalizable} 
if it can be represented by a derivation 
$d\in \Der(A)$ which acts diagonalizably on $A$, with respect to some $k$-linear basis of $A$. 
  More generally, we say that a subspace $\t \subseteq \HH^1(A)$ is  \emph{diagonalizable} if  
its elements can be represented by derivations which are simultaneously diagonalizable on $A$. Note that $\t$ is automatically a subalgebra of $\HH^1(A)$ since $[\t,\t]= 0$. The \emph{maximal diagonalizable subalgebras} are by definition those diagonalizable subalgebras which are maximal with respect to inclusion.
\end{definition}

Diagonalizable derivations were studied by Farkas, Green, and Marcos \cite{FGM} and by Le Meur \cite{Le}. Note that Le Meur assumes that diagonalizable derivations preserve the radical, and we do not.

With some background in Lie theory one can characterize diagonalizable subalgebras intrinsically in terms of tori; see the next section for definitions (Definition \ref{char0torusdef} in characteristic zero and Definition \ref{charptorusdef} in positive characteristic). While the next result uses  standard theory, it is significant in this paper because it will imply derived invariance (and invariance under stable equivalences of Morita type) for diagonalizable subalgebras.

\begin{proposition}\label{p_lieHH} 
The maximal tori of $\HH^1(A)$ (or $\Der(A)$ or $\Der_0(A)$) are exactly the maximal diagonalizable subalgebras of $\HH^1(A)$ (or $\Der(A)$ or $\Der_0(A)$, respectively).
\end{proposition}

\begin{proof} We use results and notation from Section \ref{Maximal_tori}. 

When $k$ has characteristic zero we use the theory of algebraic groups. Let $\Aut(A)$ be the algebraic group of $k$-linear automorphisms of $A$, and $\Aut_0(A)$ be the closed subgroup of those automorphisms which fix each idempotent $e_i$. We apply Proposition \ref{p_char0diagtorus} to the natural embeddings of $\Aut(A)^\circ$  and $\Aut_0(A)^\circ$ into ${\rm GL}(A)$, and then we apply Proposition \ref{p_char0torusbijection} and Lemma \ref{l_diag_aut_der} to deduce the desired statement for $\Der(A)=\L(\Aut(A))$ and $\Der_0(A)=\L(\Aut_0(A))$. After this, the statement for $\HH^1(A)$ follows by applying Proposition \ref{p_char0torussurjection} to the surjection $\Aut(A)^\circ\to \Out(A)^\circ$, and then applying Proposition \ref{p_char0torusbijection} with the isomorphism $\L(\Out(A))\cong \HH^1(A)$ of \cite[Th\'eor\`eme 1.2.1.1]{Stra2}.

In positive characteristic we use the theory of restricted Lie algebras. For $\Der(A)$ and $\Der_0(A)$ the statement follows from Proposition \ref{psemisimple} by embedding into $\gl(A)$. Then for $\HH^1(A)$ the statement follows from Proposition \ref{t_charpsurjection} applied to the surjection $\Der(A)\to \HH^1(A)$.
\end{proof}

\begin{lemma}\label{l_Der0tor}
For every maximal torus $\t\subseteq \HH^1(A)$ there is a maximal torus $\t'\subseteq \Der_0(A)$ whose image under the surjection $\Der_0(A)\to \HH^1(A)$ is $\t$.
\end{lemma}

\begin{proof}
The proof is very similar to that of Proposition \ref{p_lieHH}: in characteristic zero we use the surjection $\Aut_0(A)^\circ\to \Out(A)^\circ$, and in positive characteristic we use the surjection $\Der_0(A)\to \HH^1(A)$.
\end{proof}

\begin{lemma}\label{l_diag_pres}
If $\t \subseteq \HH^1(A)$ is a diagonalizable subalgebra then, up to inner derivations, the elements of $\t$ are simultaneously diagonalizable with respect to a basis of paths for some (possibly non-admissible) presentation $\nu\colon kQ\to A$. 
\end{lemma}

\begin{proof}
It follows from Proposition \ref{p_lieHH} and Lemma \ref{l_Der0tor} that $\t$ is the image of a diagonalizable subalgebra $\t'\subseteq \Der_0(A)$, so we may represent the elements of $\t$ by derivations $\delta$ such that $\delta(e_i)=0$ for all $i=1,...,n$, and all of which are diagonal with respect to some basis $\mathcal{B}$ of $A$. 

Firstly, by replacing $\mathcal{B}$ with a subset of $\{e_jbe_i : b\in \mathcal{B}, \ i,j=1,...,n\}$, we can assume that each element of $\mathcal{B}$ satisfies $e_jbe_i\neq 0$ for a unique pair $i,j$. By removing elements from $\{e_1,...,e_n\}\cup \mathcal{B}$, we can also assume that $\{e_1,...,e_n\}\subseteq \mathcal{B}$. Now let $Q_0=\{e_1,...,e_n\}$ and let $Q_1\subseteq \mathcal{B}\smallsetminus \{e_1,...,e_n\}$ be a subset such that $Q_0\cup Q_1$ descends to a basis of $A/{\rm rad}(A)^2$. The induced homomorphism $\nu\colon kQ\to A$ is surjective {modulo} ${\rm rad}(A)^2$, and since ${\rm rad}(A)$ is nilpotent it follows by induction that $\nu$ is surjective.

To end the proof, $A$ admits a basis consisting of paths in the quiver $Q$, and it follows from the Leibniz rule that each $\delta$ as above is diagonal with respect to this basis.
\end{proof}

\begin{definition} With Lemma \ref{l_diag_pres} in mind, we say that a derivation $\delta\in \Der(A)$ is diagonal with respect to a presentation $\nu \colon kQ\to A$ if $\delta(e)=0$ for each $e\in \nu(Q_0)$ and $\delta(a)\in ka$ for each $a\in \nu(Q_1)$. Equivalently, $\delta$ is diagonal with respect to a basis of paths induced by $\nu$. We also define
\[
\t_\nu^{\rm Der}=\left\{\delta\colon A\to A ~:~ \delta\text{ is a derivation diagonal with respect to } \nu \right\},
\]
and the image of $\t_\nu^{\rm Der}$ in $\HH^1(A)$ is denoted by $\t_\nu^{\rm HH}$.
\end{definition}

\begin{proposition}
\label{p_diagtorus}
Every maximal torus of $\HH^1(A)$ is of the form $\t_\nu^{\rm HH}$ for some presentation $\nu\colon kQ\to A$.
\end{proposition}

\begin{proof} This follows from Proposition \ref{p_lieHH} and Lemma \ref{l_diag_pres}.
\end{proof}

\begin{remark}
\label{r_admisschar}
The derivations in $\t^{\rm Der}_\nu$ will preserve the radical of $A$ if $\nu$ is an admissible presentation, and conversely, for derivations preserving the radical, Lemma \ref{l_diag_pres} can be adapted to produce an admissible presentation.  If $k$ has characteristic zero then all derivations preserve the radical \cite[Theorem 4.2]{HH}, 
so in this context we can restrict to admissible presentations in all our theorems.  In positive characteristic this is false (see Example \ref{e_JW}), so it will be important to consider non-admissible presentations. 
Several authors have instead considered the subalgebra $\HH^1_{\rm rad}(A)$ of $\HH^1(A)$ spanned by derivations which do preserve the radical \cite{ER,LR}.
\end{remark}

\section{Maximal tori}
\label{Maximal_tori}

In this section we recall what we need to know about tori in Lie algebras in order to complete the proof of Proposition \ref{p_diagtorus}. A satisfactory theory requires some extra structure, which depends on the characteristic of the given field: in characteristic zero we need an algebraic group whose tangent space at the identity is the given Lie algebra, and in positive characteristic we need a restricted structure.

The reader may wish to only skim this section for the needed Lie algebra facts, and can consult the given references for a fuller picture.

\subsection{Tori in characteristic zero}

For the rest of this subsection $k$ is an algebraically closed field of characteristic zero (the results here require this characteristic assumption). In this context, we introduce tori through the theory of algebraic groups.

The algebraic group that we are most interested in is $G=\Out(A)$, the group of outer automorphisms of a finite dimensional algebra $A$ over $k$; see \cite{HZS}. Note that $\Out(A)$ is affine since it is a quotient of the affine algebraic group $\Aut(A)$ of all $k$-linear automorphisms of $A$ \cite[Theorem 11.5]{Hum}.

\begin{definition}
\label{char0torusdef}
 An algebraic group $T$ over $k$ is called a \emph{torus} if it is isomorphic as an algebraic group to ${\rm D}_n(k)$, the group of invertible diagonal $n\times n$ matrices over $k$, for some $n$. If $G$ is any algebraic group, then a torus in $G$ is a closed algebraic subgroup $T\subseteq G$ which is (abstractly) a torus.
 \end{definition}

\begin{proposition}[{\cite[Theorem 12.12]{Mi}}]
\label{p_char0diagtorus}
Let $T$ be a connected algebraic group over $k$, and let $\phi\colon T\to {\rm GL}_n(k)$ be any embedding into the group of invertible  $n\times n$  matrices over $k$. Then $T$ is (abstractly) a torus if and only if $\phi(T)$ consists of simultaneously diagonalizable matrices.
\end{proposition}

\begin{proposition}[{\cite[Proposition 17.20]{Mi}}]
\label{p_char0torussurjection}
Let $\phi\colon G\to G'$ be a surjective homomorphism of connected algebraic groups over $k$. 
\begin{itemize}
    \item If $T$ is a maximal torus of $G$, then $\phi(T)$ is a maximal torus of $G'$.
    \item If $T'$ is a maximal torus of $G'$ and $T$ is a maximal torus of $\phi^{-1}(T')$, then $T$ is a maximal torus of $G$.
\end{itemize}
\end{proposition}

 We use the notation $\L(G)$ for the Lie algebra attached to an algebraic group $G$ over $k$; that is, the tangent space at the identity with bracket induced by differentiating the commutator operation of $G$.  See \cite[Chapter 10]{Mi} for more information.   The connected component of the identity of $G$ will be denoted $G^\circ$, and we note that $\L(G^\circ)=\L(G)$ by definition.  
 In the next definition, we use the fact that a closed algebraic subgroup $T\subseteq G$ induces an inclusion of Lie algebras $\L(T)\subseteq \L(G)$.

 \begin{definition}
\label{char0Lietorusdef}
Let $\g=\L(G)$ be the Lie algebra of an algebraic group $G$ over $k$. A Lie subalgebra $\t\subseteq \g$ is a \emph{torus} if $\t=\L(T)$ for some torus $T\subseteq G$.
\end{definition}

\begin{proposition}
\label{p_char0torusbijection}
Let $\g=\L(G)$ be the Lie algebra of a connected algebraic group $G$ over $k$. The assignment $T\mapsto \L(T)$ induces a bijection between the maximal tori in $G$ and the maximal tori in $\g$.
\end{proposition}

\begin{proof}
By \cite[Theorem 13.1]{Hum} the assignment $H\mapsto \L(H)$ induces a bijection between connected closed subgroups of $G$ and Lie subalgebras of $\g$ of the form $\L(H)$. It restricts to a bijection between tori by Definition \ref{char0Lietorusdef}, and since it is inclusion preserving, it restricts further to a bijection between maximal tori.
\end{proof}

 \begin{definition}
 Let $\g=\L(G)$ be the Lie algebra of an algebraic group $G$ over $k$. The maximal toral rank of $\g$ is 
\[
\mtrank(\g):=\max\{\dim \t ~:~ \t\subseteq  \g\text{ is a  torus}\}.
\]
\end{definition}

\begin{remark}
\label{conjugtori}
Let $T$ and $T'$ be two maximal tori of $G=\Out(A)^{\circ}$. Since all maximal tori of $G$ are conjugate by \cite[Theorem 17.10]{Mi}, we obtain an automorphism of $\HH^1(A)$ which sends $\L(T)$ to $\L(T')$.  In particular the dimensions of all maximal tori are the same.
\end{remark}

Now we turn to the case when $G=\Aut(A)$.

\begin{lemma}\label{l_diag_aut_der}
Let $A$ be a finite dimensional algebra over $k$, and let $\mathcal{B}$ be a basis of $A$. Then a closed connected algebraic subgroup $T\subseteq \Aut(A)^\circ$ consists of automorphisms diagonalizable with respect to $\mathcal{B}$ if and only if $\L(T)\subseteq \Der(A)$ consists of derivations diagonalizable with respect to $\mathcal{B}$.
\end{lemma}

\begin{proof}
Let $D_{\mathcal{B}}(A)$ be the closed subgroup of ${\rm GL}(A)$ consisting of all $k$-linear automorphisms of $A$ (ignoring the algebra structure) which are diagonal with respect to $\mathcal{B}$. A standard calculation shows that $\L(D_{\mathcal{B}}(A))$ is the set of all $k$-linear endomorphisms of $A$ which are diagonal with respect to $\mathcal{B}$. So we must show that $T\subseteq D_{\mathcal{B}}(A)$ if and only if $\L(T)\subseteq \L(D_{\mathcal{B}}(A))$, and this is part of \cite[Theorem 12.5]{Hum}.
\end{proof}

\begin{remark}
There has been recent interest in investigating when $\HH^1(A)$ is a solvable Lie algebra \cite{RSS,LR,ER}. In characteristic zero the Lie algebra of a connected algebraic group is solvable if and only if the connected algebraic group is solvable \cite[Proposition 2.7]{TY}; in particular $\Out(A)^{\circ}$ is solvable if and only if $\HH^1(A)=\L(\Out(A)) $ is solvable; see \cite[Th\'eor\`eme 1.2.1.1]{Stra2} for this equality.

In the solvable context, by \cite[Theorem 16.33]{Mi}, we have that $\HH^1(A)$ decomposes as a semidirect product $\t\ltimes \n$ where $\t$ is a maximal torus and $\n$ is a nilpotent subalgebra of $\HH^1(A)$ (see  also \cite[Proposition 26.4.5]{TY}). 
Therefore, in characteristic zero we can generalize \cite[Proposition 2.7]{Le}: we obtain a semidirect product  decomposition of $\HH^1(A)$ when the 
underling graph of the Gabriel quiver of $A$ is simply directed \cite{LR,RSS} or  when $A$ is a quantum complete intersection with quantum parameters  $q_{ij} \neq 1$ for all $i < j$; see  \cite[Proposition 4.20]{RSS}.  Note that the Lie structure of $\HH^1(A)$ can also be non-solvable when $A$ is a quantum complete intersection \cite[Proposition 4.21]{RSS}.

\end{remark}

\subsection{Tori in positive characteristic}

In this subsection we assume that $k$ is an algebraically closed field of positive characteristic $p$. 
For further background on restricted Lie algebras see  \cite[Chapter 2]{SF}.

\begin{definition} 
Let $\g$ be a Lie algebra over $k$. The \emph{adjoint representation} of  $\g$ is the map $\mathrm{Ad}\colon \g \to\mathrm{End}_k( \g)$ defined by $\mathrm{Ad}(x)(y)=[x,y]$.

We say that $\g$ is  a 
\emph {restricted Lie algebra}  if there is a map $[p]:\g\to \g$, called the $p$-\emph{power operation}, such that:
\begin{itemize}
 \item $\mathrm{Ad}(x^{[p]})(y)= \mathrm{Ad}(x)^{p}(y)$
\item $(\lambda x)^{[p]}=\lambda^p x^{[p]}$
\item $(x+y)^{[p]}=x^{[p]}+y^{[p]} + \sum_{i=1}^{p-1} \frac{1}{i} {s_i(x,y)}$
\end{itemize}
 for all $x,y \in \g$ and $\lambda \in k$.  Here, the element $s_i(x,y)$ is the coefficient of $t^{i-1}$ in
\[
\mathrm{Ad}^{p-1}(tx+y)(x)=[tx+y,\dots, [tx+y,x]\dots].
\]
Suppose $(\g_1, [p]_1)$ and  $(\g_2, [p]_2)$ are two restricted Lie algebras over $k$. A homomorphism of Lie algebras $\phi\colon \g_1 \to \g_2$ is called \emph{restricted} if $\phi(x^{[p]_1})=\phi(x)^{[p]_2}$ for every $x\in \g_1$. In particular, a subalgebra $\h \subseteq \g$ is called a \emph{restricted subalgebra} if $x^{[p]}\in \h$ for every $x\in \h$, and a representation $\phi\colon \g \to \gl_n(k)$ is called \emph{restricted} if $\phi(x^{[p]})=\phi(x)^p$ for any $x\in \g$.
\end{definition}

An important example is the Lie algebra $\gl_n(k)$ of all $n\times n$ matrices over $k$; here the $p$-power operation takes a matrix $x$ to its $p$th power $x^p$. If $A$ is a finite dimensional algebra over $k$, then the $p$th power of any derivation is again a derivation, and this gives $\Der(A)$ the structure of a restricted Lie algebra. This passes to the quotient $\Der(A)/\Inn(A)$ to make $\HH^1(A)$ a restricted Lie algebra as well, cf.\ \cite{BR}.

\begin{definition}\label{charptorusdef}

Let $\g$ be a restricted Lie algebra over $k$. An element $x\in \g$ is called \emph{semisimple} if $x$ it belongs to the restricted subalgebra generated by $x^{[p]}$. We say $x$ if \emph{toral} if it satisfies  $x^{[p]}=x$. A restricted subalgebra $\t\subseteq \g$ is called \emph{torus} if 
\begin{itemize}
    \item $\t$ is abelian
    \item $x$ is semisimple for every $x\in\t$
\end{itemize}
\end{definition}

 \begin{definition}
The maximal toral rank of a restricted Lie algebra $\g$ is
\[
\mtrank(\g):=\max\{\dim \t ~:~ \t\subseteq  \g\text{ is a  torus}\}.
\]
In the literature $\mtrank(\g)$ is also denoted by $\MT(\g)$  \cite{Strade}, 
by $\mt(\g)$  \cite{BF}, and by $\mu(\g)$ \cite{Chang}.
\end{definition}

\begin{remark}
 Demuskin proved that if $\g$ is a restricted Lie algebra of Cartan type, then all maximal tori of $\g$ have the same dimension  \cite[Sec.~7.5]{Strade}. However, this is false for a general restricted Lie algebra.
\end{remark}

\begin{proposition}
\label{toralsemisimple} \label{psemisimple}
Let $\g$ be a restricted Lie algebra over $k$, and let $\phi\colon \g\to \gl_n(k)$ be any faithful $p$-representation. Then a $p$-subalgebra $\t\subseteq \g$ is a torus if and only if $\phi(\t)$ consists of simultaneously diagonalizable elements.
\end{proposition}

\begin{proof}
Since $\phi$ is injective, it follows from definition \ref{charptorusdef} that  $\t\subseteq \g$ is a torus if and only if $\phi(\t)\subseteq \gl_n(k)$ is a torus. By \cite[Proposition 3.3]{SF} a matrix $x\in \gl_n(k)$ is semisimple in the restricted Lie algebra $(\gl_n(k),p)$ if and only if $x$ is diagonalizable. Thus $\phi(\t)$ is a torus exactly when $[\phi(\t),\phi(\t)]=0$ and each $x\in \phi(\t)$ is diagonalizable. This in turn holds if and only if the elements of $\phi(\t)$ are simultaneously diagonalizable.
\end{proof}

\begin{proposition}[\cite{SF}]
\label{t_charpsurjection}
Let $\phi\colon \g \to \g'$ be a a surjective, restricted homomorphism of finite dimensional restricted Lie algebras over $k$.
\begin{itemize}
    \item If $\t$ is a maximal torus of $\g$, then $\phi(\t)$ is a maximal torus of $\g'$.
    \item If $\t'$ is a maximal torus of $\g'$ and $\t$ is a maximal torus of $\phi^{-1}(\t')$, then $\t$ is a maximal torus of $\g$.
\end{itemize}
\end{proposition}

Later we will also discuss examples of $p$-nilpotent restricted Lie algebras:

\begin {definition}
A restricted Lie algebra $\g$ is \emph{$p$-nilpotent} if every element $x$ is $p$-nilpotent, that is, if there exists an integer $n$ such that $x^{[p]^n}=0$, for all $x \in \g$. 
\end{definition}

\section{Tori and the fundamental group}
\label{s_toripi}

Let $A$ be a finite dimensional $k$-algebra with a presentation $\nu\colon kQ\to A$, and write $I=\ker(\nu)$. Assume that $Q$ is connected. Denote by $e_1,...,e_n$ the vertices of $Q$,  
with  $e_1$ as our chosen base point, so that $\pi_1(Q,I)=\pi_1(Q,I,e_1)$ by definition. As above the group of additive characters on $\pi_1(Q,I)$ will be written $\pi_1(Q,I)^\vee=\Hom(\pi_1(Q,I),k^+)$.

We start this section by recalling a construction from \cite[3.2]{AP}. We first need to chose for each $e_i$ a walk $w_i$ from
$e_1$ to $e_i$, with $w_1$ being the trivial walk at $e_1$, 
and we write $W=\{w_i\}$ (in \cite{FGM} these are called ``parade data''). With this we define the map
\[
\theta_{\nu,W}\colon\pi_1(Q,I)^\vee\longrightarrow \Der(A)
\]
by the rule $\theta_{\nu,W}(f)(\nu(p))=f(w^{-1}_jpw_i)\nu(p)$, where $p$ is a path from $e_i$ to $e_j$; see \cite[Section 3]{DS} for further details. According to \cite[Corollary 3]{DS} the induced map
\[
\theta_\nu\colon\pi_1(Q,I)^\vee\longrightarrow \HH^1(A)
\]
does not depend on $W$. If $Q$ is not connected, then following the convention in Definition \ref{homotopydef}, $\theta_\nu$ is the sum of the corresponding maps defined by restricting to each of the connected components of $Q$.

The derivation $\theta_{\nu,W}(f)$ is by definition diagonal with respect to the presentation $\nu$, and therefore the image of $\theta_\nu$ is a torus in
$\HH^1(A)$. The next result is a converse to this fact.

Farkas, Green and Marcos prove that \emph{certain} diagonalizable derivations are always in the image of some $\theta_{\nu}$; see \cite[Theorem 3.2]{FGM}. A similar result is proven by de  la  Pe\~na  and Saor\'in \cite[Lemma 3]{DS}, but using a different definition of the fundamental group. In the case that $A$ is a triangular algebra, that is, $Q$ does not contain directed cycles, Le Meur proves in \cite[Proposition 2.6]{Le} that 
the map $\theta_{\nu}$ is surjective onto the diagonalizable subalgebra $\t_\nu^{\HH}$ of $\HH^1(A)$. 
The following theorem generalizes these results.

\begin{theorem}
\label{maxdiag}
Let $A$ be a finite dimensional algebra over an algebraically closed field $k$. Every maximal torus of $\HH^1(A)$ is of the form $\Im (\theta_{\nu})$ for some presentation $\nu\colon kQ \to A$.
\end{theorem}

\begin{proof}
If $A$ splits into a direct product of connected algebras $A_1\times \cdots \times A_\ell$, then $\HH^1(A)=\HH^1(A_1)\times \cdots\times \HH^1( A_\ell)$, 
and $\theta_\nu$ respects this decomposition, so we may assume that $Q$ is connected. 

By Proposition \ref{p_diagtorus} every maximal torus is of the form $\t_\nu^{\rm HH}$ for some presentation $\nu\colon kQ \to A$. Hence it suffices to show that $\Im (\theta_{\nu})= \t_\nu^{\HH}$. At the level of derivations, it is clear that $\Im (\theta_{\nu,W})\subseteq \t_\nu^{\Der}$, and we must show conversely that 
$ \t_\nu^{\Der}\subseteq \Im (\theta_{\nu,W})+ \mathrm{Inn}(A)$. 
For any diagonal derivation $\delta\in \t_\nu^{\Der}$ we will explicitly produce a map $f\colon \pi_1(Q,I)\to k$ for which $\theta_{\nu,W}(f)=\delta$ modulo $\mathrm{Inn}(A)$.

We first define $f$ on the set of all walks as follows: Let $p$ be a walk involving arrows $\alpha_1,...,\alpha_n$ and formal inverses of arrows $\beta_1^{-1},...,\beta_m^{-1}$. Since $\delta $ is diagonal, there are scalars $a_1,...,a_s$ and $b_1,...,b_t$ such that  $\delta(\alpha_i)=a_i\alpha_i$ and $\delta(\beta_j) =b_j\beta_j$ for each $i$ and $j$. With this setup we define
\[
f(p)=a_1+\cdots +a_s - b_1-\cdots -b_t.
\]
For a trivial walk $e_i$ involving no arrows this means $f(e_i)=0$. We can think of $f(p)$ as the signed trace of $\delta$ on $p$.

\medskip
\noindent
{\bf Claim 1.} $f$ is well-defined on $\pi_1(Q,I)$.
\medskip

We must check that $f$ respects the homotopy equivalence relation that defines $\pi_1(Q,I)$; this amounts to compatibility with the three conditions (\ref{hrel1}), (\ref{hrel2}) and (\ref{hrel3}) of Definition \ref{homotopydef}.

The first two conditions are straightforward to check. In the notation of (\ref{hrel1}), we have $f(\alpha^{-1} \alpha)=-f(\alpha)+f( \alpha)=0=f(e_i)$ and $f(\alpha \alpha^{-1})=f(\alpha)-f( \alpha)=0=f(e_j)$ as required. And in the notation of (\ref{hrel2}), if we assume $f(v)=f(v')$ then we have $f(wvu)=f(w)+f(v)+f(u)=f(w)+f(v')+f(u)=f(wv'u)$ as required.

To check compatibility with (\ref{hrel3}), let $r=\sum_{i=1}^{m} a_ip_i$ be a minimal relation in $I$, with all $a_i\neq 0$. Since  $p_i\sim_Ip_j$ for each pair of paths appearing in $r$, we must prove that $f(p_i)=f(p_j)$.  Since $\delta$ is diagonal with respect to $\nu$, there are scalars $c_i$ such that $\delta(p_i)=c_ip_i$ for each $i$, and from the definition of $f$ we have $f(p_i)=c_i$. Hence we must show $c_i=c_j$ for each pair $i,j$. Assume towards a contradiction that this is not the case, and without loss of generality say $c_1,...,c_{m'}$ are all equal and  $c_{m'+1},...,c_m$ are all different from  $c_1$, for some $m'<m$. By Lagrange interpolation we can find a polynomial $F(x)\in k[x]$ 
such that $F(c_i)=1$ for 
$i=1, ..., m^{\ensuremath{\prime}}$ and 
$F(c_i)=0$ for $i=m^{\ensuremath{\prime}}+1, ..., m$.  Regarding $\delta$ as a derivation on $kQ$ that preserves $I$, it 
follows that $\delta^n(I)\subseteq I$ for all $n\geq 0$. This implies  
that $F(\delta)$ is a linear endomorphism of $kQ$ such that $F(\delta)(I)\subseteq I$. 
In particular, we have 
$F(\delta)(r)\in I$ and
\[
F(\delta)(r)=F(\delta)\left(\sum a_ip_i\right)=\sum a_i F(c_i)p_i\\
=\sum_{k=1}^{m'} a_ip_i \in I.
\]
But this means $r$ is not a minimal relation, a contradiction.

This completes the proof that $f$ descends to a well-defined homomorphism $\pi_1(Q,I)\to k^+$. It remains to show

\medskip
\noindent
{\bf Claim 2.} $\theta_{\nu,W}(f)= \delta$ modulo an inner derivation.
\medskip

Recall that we have fixed a set $W=\{w_i\}$ of walks $w_i$ from $e_1$ to $e_i$. Define  $e_W=\sum_{i=1}^n f(w_i)e_i$. Then for any path $p$ from $e_i$ to $e_j$ we have
\begin{align*}
\theta_{\nu,W}(f)(p) & =f(w^{-1}_jpw_i)p\\
& = (-f(w_j)+f(p)+f(w_i))p \\
& = \delta(p)+f(w_i)p-f(w_j)p \\
& = \delta (p)-[e_W,p]
\end{align*}
Hence we have $\theta_{\nu,W}(f)=\delta-[e_W,-]$.
\end{proof}

Removing redundant generators from a presentation $\nu\colon kQ\to A$ can only make $\t^{\rm HH}_\nu$ larger. So, in lieu of admissibility, we say $\nu$ is \emph{minimal} if no proper subset of $Q_0\cup Q_1$ generates $A$. If $A$ is a basic split $k$-algebra and $Q_0$ is a complete set of primitive orthogonal idempotents for $A$, then minimality implies that $Q$ is isomorphic to the Gabriel quiver of $A$.

\begin{definition}
\label{pirank}
For a finite dimensional connected $k$-algebra $A$, we set
\[
\pirank(A)=\max\{\dim_k\pi_1(Q,I)^\vee ~:~ A\cong kQ/I \textrm{ a minimal presentation} \}.
\]
If $A$ is not connected, then it splits into a direct product of connected algebras $A=A_1\times \cdots \times A_\ell$ and we set $\pirank(A)=\pirank(A_1)+\cdots +\pirank(A_\ell)$ by convention.
\end{definition}

Note that in characteristic zero $\dim_k\pi_1(Q,I)^\vee$ is equal to the rank of the abelianization $\pi_1(Q,I)^{\rm ab}$, and if $k$ has positive characteristic $p$ then instead  $\dim_k\pi_1(Q,I)^\vee$ is equal to $\rank \pi_1(Q,I)^{\rm ab}+ p\text{-}\!\rank\pi_1(Q,I)^{\rm ab}$. Thus, in general algebras of positive characteristic will have larger $\pirank$.

We have come to our main motivation for studying the fundamental group: even though it depends on a presentation, we can use Theorem \ref{maxdiag} to produce derived (and stable) invariants by considering all presentations.

\begin{corollary}\label{derinv}
For any finite dimensional algebra $A$ over an algebraically closed field $k$, we have 
\[
\mtrank(\HH^1(A))=\pirank(A).
\]
In particular, $\pirank(A)$ is a derived invariant. Among algebras which are further assumed self-injective, $\pirank(A)$ is also invariant up to stable equivalence of Morita type.
\end{corollary} 

\begin{proof}
Theorem \ref{maxdiag} implies that $\mtrank(\HH^1(A))$ is the maximum of $\dim_k\Im(\theta_\nu)$ over all presentations $\nu\colon kQ\to A$. According to \cite[Theorem 2.1]{FGM} or \cite[Corollary 3]{DS}, the map $\theta_\nu$ is injective---these references work with admissible presentations, but the proofs work for minimal presentations. Therefore, the maximum of $\dim_k\Im(\theta_\nu)$ is equal to $\pirank(A)$.

In characteristic zero, the derived invariance of $\mtrank(\HH^1(A))$ follows from the fact that $\Out(A)^{\circ}$ is a derived invariant of $A$; see \cite[Theorem 17]{HZS} or  \cite[Th\'eor\`eme 4.2]{R}. In positive characteristic, derived invariance follows from 
\cite[Corollary 2]{BR}, since   $\mtrank(\HH^1(A))$ depends only on its restricted Lie algebra structure.

Similarly, among self-injective algebras $\Out(A)^\circ$ is invariant under stable equivalences of Morita type \cite[Th\'eor\`eme 4.3]{R}, so  $\mtrank (\HH^1(A))=\mtrank \L(\Out(A))$ is as well in characteristic zero, and in positive characteristic the desired statement follows from \cite[Theorem 1]{BR}.
\end{proof}

\begin{remark}
In \cite{CRSO} an intrinsic version of the fundamental group is associated to a basic algebra $A$. We have not investigated the relation between  $\pirank(A)$ as defined in this paper and rank of the intrinsic fundamental group defined in \cite{CRSO}. However, these two numbers do not coincide in general, since in the case of a Kronecker quiver, a maximal fundamental group has rank $1$ while the intrinsic fundamental group is trivial; see \cite[Proposition 36]{CRSO2}.
\end{remark}

\begin{example}\label{e_JW} 
Let $k$ be a field of characteristic $p$ and consider the algebra $A=k[x]/(x^p)$. The Lie algebra $\HH^1(A)$ was computed in  \cite{Jac}: it admits a basis $\{\partial_x,x\partial_x,...,x^{p-1}\partial_x\}$ where $\partial_x$ is the unique $k$-linear derivation such that $\partial_x(x)=1$, and the Lie bracket and $p$-power operation are
\[
[x^i\partial_x,x^j\partial_x]=(j-i)x^{i+j-1}\partial_x,
\quad \quad  
(x^i\partial_x)^{[p]}=\begin{cases}x\partial_x & \text{if }i=1\\ 0 & \text{otherwise.}
\end{cases}
\]
This is known as the Jacobson-Witt Lie algebra; it is simple whenever $p>2$  by \cite[Theorem 1]{Jac}. 
Note that the derivation $x\partial_x$ is $p$-idempotent. More generally, an element $f=a_0+a_1 x+\dots + a_{p-1}x^{p-1}$ is a $k$-algebra generator for $A$ exactly when $a_1\neq 0$, in which case $f'=\partial_x(f)$ is invertible, and we consider the corresponding derivation $f\partial_f=f(f')^{-1}\partial_x$. This derivation is $p$-idempotent since it sends $f$ to $f$, and in particular ${\rm span}\{f\partial_f\}$ is a torus.

Let $Q$ be the quiver with one vertex and one arrow $u$, so that $kQ=k[u]$, and let $\nu_f\colon kQ\to A$ be the minimal presentation sending $u$ to $f$. Note that these presentations are non-admissible unless $a_0=0$. The kernel of $\nu_f$ is $I_f=(u^p-a_0^p)$, and one can compute \[
\pi_1(Q,I_f)=
\begin{cases}
\mathbb{Z} & a_0=0\\
\mathbb{Z}/p\mathbb{Z} & a_0\neq 0.
\end{cases}
\]
Regardless of $a_0$, this means $\pi_1(Q,I_f)^\vee = k$ is one-dimensional, generated by a class dual to the loop $u$. Finally, the definition of $\theta_{\nu_a}$ gives $\Im(\theta_{\nu_f})={\rm span}\{f\partial_f\}$, and we see every nonzero torus is generated by some $f\partial_f$, with almost all of them coming from non-admissible presentations.
\end{example}

\section{Applications}
\label{s_examples}

The first Hochschild cohomology group provides a bridge between finite dimensional algebras and Lie algebras. 
In this section, we consider some families of finite dimensional algebras and we ask some question on both sides of this bridge. As a consequence, we obtain some well-known theorems and some new results which a priori are not related with the Lie structure of the first Hochschild cohomology.

One natural question to ask is the following: what is the relation between  $\HH^1(A)$ and the first Betti number  $\beta_1(Q)$ of a finite dimensional algebra $A$ with  Gabriel quiver $Q$?

\begin{theorem}
\label{rankbetti}\label{monomialbetti}
Let $A$ be a finite dimensional algebra over an algebraically 
closed field $k$, with Gabriel quiver $Q$, then \[
\mtrank(\HH^1(A))\leq \beta_1(Q).
\]
If $A$ is a monomial (or semimonomial) algebra then equality holds.
\end{theorem}

\begin{proof}
Replacing $A$ with its basic algebra, this follows from Lemma \ref{ineq} and Corollary \ref{derinv}.
\end{proof}

\begin{remark}
In characteristic zero the converse of Theorem \ref{rankbetti} holds, that is, $\mtrank(\HH^1(A)) = \beta_1(Q)$ if and only if $A$ is semimonomial; see Remark   \ref{r_semimon_converse}. In positive characteristic one can also obtain a perfect converse using $p$-semimonomial ideals, along the lines explained in Remark \ref{r_semimon_converse} (but we caution that when equality holds the guaranteed $p$-semimonomial presentation may be non-admissible).

Compare this with \cite[Theorem 1.2]{GAS}, which characterizes semimonomial algebras $A$ in terms of the rank of the algebraic group $\Aut(A)^\circ$.
\end{remark}

\begin{remark}
Antipov and  Zvonareva compute the rank of the maximal torus of $\Out(A)^\circ$ for any symmetric stably biserial
algebra $A$ 
in terms of the corresponding Brauer graph \cite[Theorem 1.1]{AZ}.  They use this to prove that Brauer graph algebras are closed under derived 
equivalence  \cite[Corollary 1.3]{AZ}.  
The rank they compute is less than the maximum allowed in Theorem \ref{rankbetti} by $v-1$,  
where $v$ is the number of vertices in the Brauer graph of $A$.
\end{remark}

\subsection{Monomial and semimonomial algebras} Bardzell and Marcos proved that if $A = kQ/I$ is a finite dimensional monomial algebra  such that $\HH^1(A) = 0$, then $Q$ is a tree {\cite[Theorem 2.2]{BM}}. As an application we recover and generalize their result.

\begin{corollary}
Let $A = kQ/I$ be a finite dimensional  monomial (or semimonomial) algebra. Then $\mtrank(\HH^1(A))=0$ if and only if $Q$ is a tree.
\end{corollary}

\begin{proof}
The statement follows from Theorem \ref{rankbetti} because $\beta_1(Q)=0$ if and only if $Q$ is a tree.
\end{proof}

To explain the hypothesis of the theorem: it is certainly true that if $\HH^1(A)=0$, then $\mtrank(\HH^1(A))=0$. More generally $\mtrank(\HH^1(A))$ vanishes exactly when $\Out(A)^\circ$ is unipotent in chacteristic zero, and exactly when $\HH^1(A)$ is $p$-nilpotent in positive characteristic.

It was proven by Avella-Alaminos and Geiss that among gentle algebras the number of arrows is a derived invariant \cite[Proposition B]{AG}. Surprisingly, this can be extended to all monomial algebras, and even all semimonomial algebras:

\begin{theorem}
\label{monomialrank}
Let $k$ be a field, and let $A$ be a finite dimensional split $k$-algebra which is derived equivalent to a finite dimensional semimonomial $k$-algebra $B$. Then the Gabriel quiver of $A$ contains at least as many arrows as the Gabriel quiver of $B$.

\label{monarrows}
In particular, if two finite dimensional semimonomial algebras are derived equivalent, then they have the same number of arrows.
\end{theorem}

\begin{proof}
Tensoring with an algebraic closure preserves derived equivalences, and preserves the number of arrows, so we may assume that $k$ is algebraically closed.

By Theorem \ref{rankbetti} and Corollary \ref{derinv} we have  
\[
\beta_1(Q_A)\geq \mtrank(\HH^1(A))=\mtrank(\HH^1(B))=\beta_1(Q_B),
\]
where $Q_A$ and $Q_B$ are the Gabriel quivers of $A$ and $B$ respectively. The statement now follows because the number of simple modules  and the number of connected components of the Gabriel quiver are both derived invariants.
\end{proof}

The following consequence is rather surprising: 

\begin{theorem}
\label{finmono}
Let $k$ be a field. There are only finitely many monomial algebras in each derived equivalence class of finite dimensional $k$-algebras.
\end{theorem}

\begin{proof}
Within a given derived equivalence class, all monomial algebras have the same number of vertices and arrows by Theorem \ref{monarrows}, and therefore there are only finitely many possible Gabriel quivers.

We need to show that there are finitely many possible derived equivalent monomial ideals on a given quiver, and this is only an issue if there could be arbitrarily long defining relations. If $A$ is an algebra with a complete set of orthogonal idempotents $e_1,...,e_n$, then the number $\sum \dim_ke_iAe_i$ is a derived invariant of $A$, since it is the trace of the Cartan matrix of $A$, and the similarity class of the Cartan matrix is a derived invariant (see the proof of \cite[Proposition 1.5]{BS}). It follows from this that the lengths of the defining relations of monomial algebras are uniformly bounded within any given derived equivalence class.
\end{proof}

\subsection{Commutative monomial algebras and and quantum complete intersections}\label{s_semimon}
Here we discuss two important classes of semimonomial algebras. The first example is from commutative algebra.

\begin{corollary}
\label{c_comm}
Let $A=k[x_1,...,x_n]/I$ be a finite  dimensional commutative $k$-algebra. Then $\mtrank(\HH^1(A))\leq n$ and equality holds if  $I$ is a monomial ideal.
\end{corollary}

\begin{proof}
We present $A$ using the quiver with a single vertex and loops $ x_1,...,x_n$. The kernel of this presentation is semimonomial if and only if $I$ is monomial (in the commutative sense), so the statement follows from  Theorem \ref{rankbetti}.
\end{proof}

One can establish a converse to Corollary \ref{c_comm}, depending on the characteristic of $k$, by using the observation from Remark \ref{r_semimon_converse}.

For the next example fix a sequence $m_1,...,m_n$ of positive integers and a collection of elements $q=\{q_{ij}\in k^\times ~:~1\leq i <j\leq n\}$. Then the algebra 
\[
A_q=k\langle x_1,...,x_n\rangle/\left(x_i^{m_i},\  x_ix_j-q_{ij}x_jx_i\right)
\]
is known a \emph{quantum complete intersection}. 
Many authors have studied the Hochschild cohomology of quantum complete intersections;  for example \cite{BKL,BE,EH,O}. In \cite{RSS} the authors proved in many cases that the first Hochschild cohomology is a solvable Lie algebra. In \cite{BW} the authors explain how to explicitly compute the Lie algebra structure of the Hochschild cohomology for any number of variables. According to \cite[Theorem 1.1 (iii)]{BKL}, when $k$ has characteristic $p$ and $n=2$ and $q_{12}$ has order dividing $p-1$, we have $\mtrank(\HH^1(A_q))=2$ (in this context $A_q$ is the basic algebra of a non-principal block of a finite group). We generalize this result to arbitrary quantum complete intersections.

\begin{corollary}
\label{qci}
Let $A_q$ be a quantum complete intersection with $n$ variables. Then we have  $\mtrank(\HH^1(A_q))=n$.
\end{corollary}

\begin{proof}
Since $A_q$ is semi-monomial, the statement follows from Theorem \ref{rankbetti}.
\end{proof}

\subsection{Simply connected algebras}

Let $A$ be a finite dimensional $k$-algebra over an algebraically closed field $k$. Let $Q$ be the Gabriel 
quiver of $A$. If the quiver 
$Q$ is acyclic, then the algebra $A$ is called  \emph{triangular}. In  \cite{AS}, the authors define
a \emph{simply connected algebra} to be a  triangular algebra with no 
proper Galois covering, or equivalently, with  trivial 
fundamental groups for every admissible quiver presentation. We note that by \cite[Lemma 2.6]{LR}, if $A$ is triangular then all derivations preserve the radical of $A$, and so we may restrict attention to admissible presentations (see Remark \ref{r_admisschar}).

In \cite{S}, Skowroński asked for which triangular algebras $A$ we have that  $A$ is simply connected if and only if  $\HH^1(A) = 0$. This problem has motivated several results: see \cite{ABL,BL,LeM2} for example. The following theorem has a corollary that constrains which Lie algebras can be obtained as $\HH^1(A)$ of a simply connected algebra $A$.

\begin{theorem}
\label{nilpotent}
Let $A$ be a finite dimensional algebra over an algebraically closed field $k$.
\begin{itemize}
    \item If $k$ has characteristic zero, then $\Out(A)^{\circ}$ is unipotent  if and only if $\pi_1(Q,I)^{\rm ab}$ is finite for any minimal presentation $A\cong kQ/I$.
    \item If $k$ has positive characteristic $p$, then $\HH^1(A)$ is $p$-nilpotent if and only if $\pi_1(Q,I)^{\rm ab}$ finite and $p$-torsion free for any minimal presentation $A\cong kQ/I$.
\end{itemize}

\end{theorem}

\begin{proof}
We divide the proof depending of the characteristic of $k$.

We start with characteristic zero:  if $\pi_1(Q,I)^{ab}$ is always finite, then by Theorem \ref{maxdiag}  all maximal tori of $\HH^1(A)$ are zero. Therefore the maximal toral rank of $\Out(A)^{\circ}$ is zero. By \cite[Theorem 20.1]{Mi} we have that $\Out(A)^{\circ}$ is unipotent. Conversely, if $\Out(A)^{\circ}$ is unipotent then all tori of $\HH^1(A)$ are zero. By the injectivity of  $ \theta_{\nu}$, all duals of the fundamental groups are zero, hence $\pi_1(Q,I)^{ab}$ is finite for any minimal presentation $A\cong kQ/I$.

If the field has positive characteristic and if $\pi_1(Q,I)^{ab} \otimes_{\mathbb Z} \mathbb{F}_p$ is zero for every minimal presentation, then every torus is zero  by Theorem \ref{maxdiag}.  Hence there are no semi-simple elements, since the one-dimensional Lie algebra spanned by a semi-simple element is toral. By the Jordan-Chevalley-Seligman decomposition it follows that every element is $p$-nilpotent, hence $\HH^1(A)$ is $p$-nilpotent. Conversely, if $\HH^1(A)$ is $p$-nilpotent then there are no semi-simple elements. By the injectivity of  $ \theta_{\nu}$, the dual of the fundamental group is trivial. Therefore $\pi_1(Q,I)^{ab} \otimes_{\mathbb Z} \mathbb{F}_p$ is  zero for every minimal quiver presentation.
\end{proof}

\begin{corollary}
\label{HH1nilp}
If $A$ is a simply connected finite dimensional algebra over an algebraically closed field of any characteristic, then $\HH^1(A)$ is nilpotent.
\end{corollary}

\begin{proof}
In  characteristic zero, by the previous theorem it follows that $\Out(A)^{\circ}$ is unipotent, therefore $\Out(A)^{\circ}$ is  nilpotent by \cite[Proposition 14.21]{Mi}, and  $\HH^1(A)$ is nilpotent by  \cite[Corollary 24.5.13]{TY}.
In positive characteristic the statement follows from Engel's theorem, that every $p$-nilpotent finite dimensional restricted Lie algebra is nilpotent.
\end{proof}

\begin{remark}
In light of Theorem \ref{nilpotent}, one might approach Skowroński's question for a specific family of algebras by first showing that $A$ is simply connected if and only if every maximal torus  of $\HH^1(A)$ is zero and then by showing that the last condition is equivalent to $\HH^1(A)=0$. The  obstruction to the first step  is in understanding the relation between the fundamental group and its abelianization. Note that for the second step there might be shortcuts (an arrow parallel to a path that does not contain the arrow) and multiple arrows, which might lead to outer derivations 
which are not diagonalizable.
\end{remark}

\subsection{Group algebras} In the setting of Example \ref{e_group}, when $\nu\colon kQ/I\cong kG$ is a group algebra and $\pi_1(Q,I)\cong G$, the map $\theta_\nu\colon \pi_1(Q,I)^\vee\to \HH^1(kG)$ coincides with the well-known embedding ${\rm H}^1(G,k)\to \HH^1(kG)$ of group cohomology into Hochschild cohomology. It follows that the toral rank of $\HH^1(kG)$ is always at least $\dim_k {\rm H}^1(G,k)$.

We note that group algebras of $p$-groups have the opposite behaviour of simply connected algebras. In \cite[Proposition 10.7]{BKL} the authors prove that $\HH^1(kP)$ is never nilpotent if $kP$ is the group algebra of a $p$-group $P$. Combining this with Corollary \ref{HH1nilp}, it follows that $p$-groups are never simply connected. We complement this with the following statement.

\begin{proposition}
\label{p_pgroup}
Let $P$ be a $p$-group, and $k$ an algebraically closed field of characteristic $p$. Then $\mtrank(\HH^1(kP))$ is the minimal number of generators of $P$.
\end{proposition}

\begin{proof}
The Gabriel quiver is a single vertex with $n$  loops, where $n$ is equal to $\dim_k {\rm H}^1(G,k)=p\text{-}\!\rank(G/\Phi(G))$, which is the minimal number of generators of $P$. Hence $\mtrank(\HH^1(kP))\leq n$ by Theorem \ref{rankbetti}.  Conversely, using Example \ref{e_group} and Corollary \ref{derinv},  $\mtrank(\HH^1(kP))\geq \dim_k \Hom(G,k^+)= \dim_k {\rm H}^1(G,k) = n$.
\end{proof}

As a consequence we have that if $B$ is a block of a finite group with defect group $P$, then in many cases $\mtrank(\HH^1(B))=\mtrank(\HH^1(kP))$. This is true if $P$ is cyclic, or if $B$ is nilpotent, or if $B$ is a block with normal abelian defect group $P\cong \mathbb{Z}/p^{n_1}\mathbb{Z}\times\dots \times \mathbb{Z}/p^{n_r}\mathbb{Z}$, abelian inertial quotient and up to isomorphism a unique simple module. The equality for the first case holds since  the block is derived equivalent to a symmetric Nakayama algebra and for the second case since nilpotent blocks are Morita equivalent to the group algebra of their defect groups. For the last family of algebras, by Theorem 1.1 in \cite{HK} we have that the basic algebra of $B$ is isomorphic
to a quantum complete intersection. By Corollary \ref{qci} and  Proposition \ref{p_pgroup} we have $\mtrank(\HH^1(B))=\mtrank(\HH^1(kP))$.

This is related with the following open question: let $G$ be a finite group and assume the characteristic of the field $k$ divides the order of $G$. What is the (restricted) Lie algebra structure of $\HH^1(kG)$?  Using the classification of finite simple groups,  it is known that $\HH^1(kG)\neq 0$. Related with this question, Linckelmann asked in \cite{L2} if the first Hochschild cohomology group of every non-semisimple block of a group algebra is nonzero.

\subsection{Kronecker chains and Beilinson algebras}
In this section we compute the dual  fundamental groups for two families of finite dimensional algebras which have a similar behaviour: quotients of path algebras of Kronecker chains by standard relations,  and Beilinson algebras.

Let $A$ be a non-wild finite dimensional algebra. Kronecker chains were introduced in \cite{ER} to express $\HH^1_{\rm rad}(A)$, the Lie subalgebra of $\HH^1(A)$ consisting of outer derivations that preserve the  radical, as a direct sum  $\r \oplus \sl_2(k)^{\oplus m}$, where $\r$ is solvable and $m$ is the number of equivalence classes of maximal Kronecker chains with standard relations embedded in $A$.

\begin{definition}[{\cite[Definition 5.1]{ER}}]
\label{maxkrost}
Let $A=kQ/I$ be an admissible presentation. A \emph{Kronecker chain} $C$ in $A$ is sequence of pairs of arrows $(a_1, b_1), \dots , (a_n, b_n)$ in $Q$, with each pair $(a_i,b_i)$ having the same source and target, and with the target of $a_i$ equal to the source of $a_{i-1}$ for $i=1,...,n-1$. We say $C$ has \emph{standard relations} if $J=kC\cap I$ is generated by:
\begin{itemize}
\item  $a_ia_{i+1} = 0$, $b_ib_{i+1} = 0$ and $a_ib_{i+1} + b_ia_{i+1} = 0$ for $i=1,...,n-1$;
\item  if the source of $a_1$ is the target of $a_n$ then also $a_na_1 = 0$, $ b_nb_1 = 0$ and $a_nb_1 + b_na_1 = 0$.
\end{itemize}
\end{definition}

\begin{example}
Let $C$ be a Kronecker chain with standard relations embedded in a non-wild algebra $A$. By \cite[Lemma 5.2]{ER} $C$ is either a Dynkin quiver of type $A_n$ with doubled arrows, or an extended Dynkin quiver of type $\widetilde{A}_n$ with doubled arrows, or a double loop $L_2$. In these cases we compute $\pi_1(C,J)^{\rm ab}$  using the homology of the following complex $X_{\ast}$
\[\begin{tikzcd}
	{\mathbb{Z}C_0} && {\mathbb{Z}C_1} && {\mathbb{Z}P}
	\arrow["{\delta_0}"', from=1-3, to=1-1]
	\arrow["{\delta_1}"', from=1-5, to=1-3],
\end{tikzcd}\]
where $P:=\{(p,q)\ |\ p, q\ \textnormal{appear together in a minimal relation}\}$. The differentials are defined as $\delta_0(a)= \mathrm{source}(a)-\mathrm{target}(a)$ and $\delta^1((p,q))=\bar{p}-\bar{q}$, where $\bar{p}=\sum {a_i}$ for a path $p=a_1\dots a_n$. If $C$ is of type $A_n$ with double arrows, then $\ker(\delta_0)=\bigoplus\mathbb{Z}(l_i)$ where $l_i=a_i-b_i$, and $\Im(\delta_1)=\bigoplus \mathbb{Z}(l_i-l_{i+1})$. Hence $\mathrm{H}_1(X_{\ast})\cong \mathbb{Z}$ and the dimension of $\pi_1(C,J)^\vee$  is $1$. If the Kronecker chain is of type $\widetilde{A}_n$, then similarly $\mathrm{H}_1(X_{\ast})\cong \mathbb{Z}^2$ and the dimension of $\pi_1(C,J)^\vee$ is $2$. Finally, if Kronecker chain has only one vertex and two loops, then $\mathrm{H}_1(X_{\ast})\cong \mathbb{Z}\oplus \mathbb{Z}/2\mathbb{Z}$. In this last case, if the field $k$ has characteristic $2$ then  $\pi_1(C,J)^\vee$ has dimension $2$, otherwise $\pi_1(C,J)^\vee$ has dimension $1$.
\end{example}

\begin{example}
Let $\Lambda$ be the exterior algebra over a field. The 
\emph {Beilinson algebra} $b(\Lambda)$  appeared for the first time in Beilinson’s paper \cite{Be} on the
bounded derived category of projective spaces (see \cite{Chen} for the general notion of Beilinson algebra $b(A)$ of a graded algebra $A$). The authors of \cite{XZMH} compute the dimension of the Hochschild cohomology of $b(\Lambda)$. 
The Beilinson algebra has a presentation $b(\Lambda)=kQ/I$ where

\[
Q = \begin{tikzcd}
	0 & \vdots & 1 & \vdots & 2 & \dots & {n-1} & \vdots & n
	\arrow["{x_{0,0}}", curve={height=-18pt}, from=1-1, to=1-3]
	\arrow["{x_{0,n}}"', curve={height=18pt}, from=1-1, to=1-3]
	\arrow["{x_{1,0}}", curve={height=-18pt}, from=1-3, to=1-5]
	\arrow["{x_{1,n}}"', curve={height=18pt}, from=1-3, to=1-5]
	\arrow["{x_{n-1,0}}", curve={height=-18pt}, from=1-7, to=1-9]
	\arrow["{x_{n-1,n}}"', curve={height=18pt}, from=1-7, to=1-9]
\end{tikzcd}
\]
\[
\text{and} \quad I=\big( x_{t,i}x_{t+1,j}-x_{t,j}x_{t+1,i}~:~  t = 0,\dots , n-2\ \ \  i, j = 0,\dots, n \big).
\]
In order to compute $\pi_1(Q,I)^{\rm ab}$, we generalize the calculations for the Kronecker chain of type $A_n$. 
It is easy to show that $\ker(\delta_0)=\bigoplus\mathbb{Z}(l_{t,i})$ where $l_{t,i}=x_{t,i}-x_{t,i+1}$ and $\Im(\delta_1)=\bigoplus\mathbb{Z}(l_{t,i}-l_{t,i+1})$. Therefore 
$\mathrm{H}_1(X_{\ast})=\bigoplus_{i=0}^{n-1}\mathbb{Z}(l_{0,i})$. Hence $\pi_1(Q,I)^{\rm ab}\cong \mathbb{Z}^{n}$. Consequently, $\pi_1(Q,I)^\vee$ has dimension $n$.
\end{example}

\subsection{Reduced universal enveloping algebras} There is an interesting relation between the representation theory of  a finite dimensional algebras $A$ and certain quotients of the enveloping algebra of $\HH^1(A)$, namely the $\chi$-reduced universal enveloping algebras.

Let $\chi$ be a character of a restricted Lie algebra $\g$. We denote by $\u(\g,\chi)$ the $\chi$-reduced
universal enveloping algebra (see  \cite[section 3.1]{Strade} or \cite{BF}), which is the universal enveloping algebra $U(\g)$ factored by the ideal
generated by the elements $x^p - x^{[p]} - \chi(x)^p\cdot 1$ for all $x \in \g$. Note that $\mtrank(\g)$ plays an important role in
 bounding the number of nonisomorphic simple $\u(\g,\chi)$-modules. More precisely:

 \begin{conjecture}[{\cite[Conjecture 3.4]{BF}}]
  Let $\g$ be an arbitrary restricted Lie algebra, and let $\chi$ be a character of $\g$. Then there are at most $p^{\mtrank(\g)}$ nonisomorphic simple $\u(\g,\chi)$-modules.
 \end{conjecture}

For the families of restricted Lie algebras for which the conjecture has been verified (e.g.\  $\sl_2$, Witt, Jacobson-Witt, and solvable Lie algebras \cite{BF}), we can deduce from Theorem \ref{rankbetti} that the number of nonisomorphic simple $\u(\HH^1(A),\chi)$-modules is bounded by $\beta(Q_A)$. These families of Lie algebras have been studied in various articles, for example \cite{ER,LR1,LR,RSS}.   
In particular, if $A$ is a semimonomial algebra such that $\HH^1(A)$ is a solvable Lie algebra, then the projective cover of the trivial irreducible module of $\HH^1(A)$ is induced from the one dimensional trivial module of a maximal torus (see \cite{FSW}) and by Theorem \ref{rankbetti}  we have $p^{\beta(Q_A)} =p^{\mtrank\HH^1(A)}$.

\end{document}